\documentclass[12pt]{l4dc2021} 

\title[Learning-based State Reconstruction for a Scalar Hyperbolic PDE]{Learning-based State Reconstruction \\ for a Scalar Hyperbolic PDE under noisy Lagrangian Sensing}

\usepackage{graphicx}
\usepackage{times} 
\usepackage{mathtools}
\usepackage{newtxmath}

\makeatletter
\def\set@curr@file#1{\def\@curr@file{#1}} 
\makeatother

\newcommand{\R}{\mathbb{R}}

\newcommand{\C}{\mathcal{C}}
\renewcommand{\H}{\mathcal{H}}
\newcommand{\D}{\mathcal{D}}
\newcommand{\F}{\mathcal{F}}
\newcommand{\G}{\mathcal{G}}
\renewcommand{\L}{\mathcal{L}}

\DeclareMathOperator*{\Argmin}{argmin}

\DeclareMathOperator*{\ReLu}{ReLu}

\newtheorem{theo}{Theorem}

\newtheorem{prop}{Proposition}

\author{%
 \Name{Matthieu Barreau} \Email{barreau@kth.se}\\
 \Name{John Liu} \Email{johnliu@kth.se}\\
 \Name{Karl Henrik Johansson} \Email{kallej@kth.se}\\
 \addr Division of Decision and Control Systems, KTH Royal Institute of Technology Stockholm, Sweden%
}

\begin{document}

\maketitle


\begin{abstract}
    The state reconstruction problem of a heterogeneous dynamic system under sporadic measurements is considered. This system consists of a conversation flow together with a multi-agent network modeling particles within the flow. We propose a partial-state reconstruction algorithm using physics-informed learning based on local measurements obtained from these agents. Traffic density reconstruction is used as an example to illustrate the results and it is shown that the approach provides an efficient noise rejection.
\end{abstract}


\section{Introduction}

Traffic jams are a source of green gas and time-consumption \citep{ferrara2018}. One solution to lighten the effect of congestion is to control the flow of vehicles. There exist a large variety of control strategies using variable speed limits \citep{delle2017traffic}, ramp control \citep{zhang2019pi} or autonomous vehicles \citep{PIACENTINI201813,wu2017flow}. These laws require knowledge of road density everywhere. There comes the need for an efficient state reconstruction algorithm using measurements from fixed sensors but also capable of dealing with GPS measurements \citep{amin2008mobile} and Probe Vehicles (PVs) sensing \citep{herrera2010incorporation}.

This is challenging from a theoretical point of view since such a system is modeled by a cascaded partial differential equation with a multi-agents network (C-PDE-MA) such as the one derived by me\cite{borsche2012mixed}. There exist several kinds of algorithms to estimate the density and one can refer to the survey by \cite{seo2017traffic} for an overview. For fixed sensors, one can investigate the use of backstepping observers \citep{smyshlyaev2005backstepping}. However, these cannot deal effectively with nonlinearities and one has to approximate the system dynamic depending on its regime \citep{YU2019183}. Using PVs, there exist some open-loop state reconstruction algorithms proposed for instance by \cite{delle2019traffic,barreau2020dynamic,cicic2020numerical}. They are copying the system and matching the boundary conditions, consequently, they are not robust to model mismatch or noise. Another recent field of research is to use machine learning tools to estimate the density \citep{matei2019inferring,huang2020physics,liu2020learning}, no matter the origin of the measurement. They are then more flexible and they lead to a good reconstruction providing substantial computer resources. 

This paper goes in that latter direction and extends the preliminary observations made in \cite{liu2020learning}. One contribution of this article is to define a partial-state reconstruction (P-SR) expressed as the solution to an optimization problem. We show that a neural network is capable of solving the noisy state reconstruction problem asymptotically and we explain how two neural networks with a dense architecture are well-suited to that problem. We also propose a training procedure that is robust with a moderate computational burden.

The outline of the paper is as follows. In Section~2, we describe the model C-PDE-MA and give some fundamental properties about the existence and regularity of solutions. We end this section by giving a general definition of state reconstruction. Similarly to what is proposed by \cite{ostrometzky2019physics}, Section~4 is dedicated to the neural network solution with a particular focus on the numerical implementation. Finally, we consider the traffic state reconstruction problem, and simulations show a good reconstruction accuracy, even in the noisy case. The conclusion enlarges the scope of the article by providing some perspectives.

\textbf{Notation:} We define $L_{loc}^1(\mathcal{S}_1, \mathcal{S}_2)$, $L^{\infty}(\mathcal{S}_1, \mathcal{S}_2)$ and $C^k(\mathcal{S}_1, \mathcal{S}_2)$ as the spaces of locally integrable functions, bounded functions and continuous functions of class $k$ from $\mathcal{S}_1$ to $\mathcal{S}_2$, respectively. A function in $C^k_p$ is piecewise-$C^k$. Let $f$ be a function defined in a neighborhood of a point $x$, we define $f(x^-)$ as $\lim_{s \to x, s < x} f(s)$. For $g$ a function defined in a neighborhood of $x$ with $g(x) \neq 0$, we write $f = o_x(g)$ if $\lim_{s \to x} \frac{f}{g}(s) = 0$.

\section{Problem Statement}

This section introduces some fundamental notions about the PDE model, the ODE model, and their coupling. It is followed by the problem considered in this paper.
Theoretical questions regarding the existence and regularity of solutions to these models are also briefly discussed.

\subsection{PDE model}

To define our problem, we will use the notation introduced by \cite{bastin2016stability}. We are interested in the following scalar hyperbolic and quasi-linear PDE for $(t,x) \in \R^+ \times \R$:
\begin{equation} \label{eq:PDE}
    \frac{\partial \rho}{\partial t} (t,x) + F(\rho(t,x)) \frac{\partial \rho}{\partial x}(t,x) = 0
\end{equation}
with initial condition $\rho(0,\cdot) = \rho_{0} \in L^{\infty}(\R, [0,1])$ is piece-wise $C^{\infty}$. In physical systems, $\rho: \R^+ \times \R \to [0, 1]$ usually refers to the normalized density. Equation~\eqref{eq:PDE} is then the macroscopic expression of the conservation of mass.

We assume that 
\begin{enumerate}
    \item $F: \R \to \R$ is of class $C^1(\R, \R)$, 
    \item The flux function $f(\rho) = \int_0^{\rho} F(s) ds$ is concave,
    \item $f(\rho) = V_f \rho + o_0(\rho)$, where $V_f > 0$ is the free flow speed.
\end{enumerate}
\begin{prop}
    There exist weak solutions to this system with the following regularity:
    \begin{equation} \label{eq:existenceSpace}
        \exists T > 0, \quad \rho \in C^0(\R^+, L^1_{loc}(\R, [0,1])) \text{ and piece-wise } C^{\infty}.
    \end{equation}
\end{prop}
The regularity result comes from a direct application of Theorems~6.2.2 and 11.3.10 by \citet{Dafermos:1315649}. The fact that $\rho \in [0, 1]$ is a consequence of the generalized characteristic methodology \cite[Chap. 4]{protter2012maximum}. Considering the Lax-E entropy condition $F(\rho^+) \leq F(\rho^-)$ guarantees that the solution is unique \cite[Theorem~14.10.2]{Dafermos:1315649}. Note that $f$ is concave, hence, the entropy condition turns out to be:
\begin{equation} \label{eq:entropy}
    \rho(t, x^-) \leq \rho(t, x^+).
\end{equation}

\subsection{Agents model}

Consider a network of $N > 1$ agents, located at $y_i \in \R$, $i \in \{1, \dots, N\}$ such that $y_1(0) < \dots < y_N(0)$. These agents are particles within a flow of local density $\rho$ defined in \eqref{eq:PDE}. In other words, $\rho(t, \cdot)$ can be seen as a probability distribution and the agents are a subset of a given realization. 

We assume first that the speed $V_i$ of agent $i$ depends only on the density at its location and that $V_i$ is a decreasing function of the density. Since the density might be discontinuous, we consider that the speed is the infimum in any neighborhood around its position. Because of the entropy condition \eqref{eq:entropy}, the velocity function writes as:
\[
    V_i(t) = \min \left\{ V (\rho(t, y_i(t)^-), V (\rho(t, y_i(t)^+)) \right\} = V (\rho(t, y_i(t)^+).
\]
The agent's dynamic is then modeled by the following ordinary differential equation (ODE):
\begin{equation} \label{eq:agents}
    \dot{y}_i(t) = V \left( \rho(t, y_i(t)^+) \right) \quad \text{ for } i \in \{1, \dots, N \}.
\end{equation}
Under the entropy condition~\eqref{eq:entropy}, one can show that the agents stay ordered in space, meaning that $y_1(t) < \dots < y_N(t)$ for any $t \in \R^+$. 
Since the number of particles between agents $1$ and $N$ is ${M(t) = \int_{y_1(t)}^{y_N(t)} \rho(t,x) dx}$ and a particle cannot overtake another one, the quantity $M$ remains constant and
\[
    \dot{M} = V(\rho_N) \rho_N - f(\rho_N) - V(\rho_1) \rho_1 + f(\rho_1) = 0,
\]
where $\rho_i(t) = \rho(t, y_i(t))$.
The previous expression must hold for any $\rho_1, \rho_N \in [0, 1]$ yielding:
\begin{equation} \label{eq:speed}
    V(\rho) = \left\{ \begin{array}{ll} f(\rho) \rho^{-1}, & \text{ if } \rho \in (0, 1], \\ V_f, & \text{ otherwise.} \end{array} \right.
\end{equation}
Because $f$ is concave and $f(0) = 0$, we get that $V$ is indeed a decreasing function of the density.
\begin{prop} There exists a unique absolutely continuous solution $y_i$ of \eqref{eq:agents} defined for $t \geq 0$.
\end{prop}
\begin{proof} From the definition of $f$, we get that $V \in C^1([0, 1], \R)$. The proposition results from the application of Carathéodory theorem \citep[Theorem~1.1, Chap~2]{coddington1955theory}. \end{proof}

\subsection{C-PDE-MA model}

In this paper, the network of agents is capable of sensing the PDE state at their locations. 
The C-PDE-MA model is illustrated in Figure~\ref{fig:generalProblem2}. It is represented by a coupling between the PDE model~\eqref{eq:PDE} and the ODE model~\eqref{eq:agents} for $t \in \R^+$ and $x \in \R$:
\begin{equation} \label{eq:C-PDE-MA}
    \hspace*{-0.1cm}
    \begin{array}{clcl}
        \text{PDE} \!\!\!& \left\{ \begin{array}{l}
            \displaystyle \frac{\partial \rho}{\partial t} (t,x) + F(\rho(t,x)) \frac{\partial \rho}{\partial x}(t,x) = 0, \\
            \rho(0,x) = \rho_{0}(x), \\
            \tilde{\rho}(t) = \C_{\tilde{y}(t)}(\rho(t, \cdot)) + n_{\rho}(t), \\
        \end{array} \right. & \quad 
        \text{MA} \!\!\!& \left\{ \begin{array}{ll}
            \dot{\tilde{y}}(t) = V(\rho(t, \tilde{y}(t)^+)),\\
            \tilde{y}(0) = \tilde{y}_0, \\
            \tilde{w}(t) = \tilde{y}(t) + n_{y}(t), 
        \end{array} \right.
    \end{array}
\end{equation}
where $n_{\rho}, n_y \in L^{\infty}(\R^+, \R^N)$ are measurement noises, $V$ is the element-wise function defined in \eqref{eq:speed}, $\tilde{y}(t) = \left[ y_1(t) \ \cdots \ y_N(t) \right]^{\top}$ and $\tilde{w}(t) = \left[ w_1(t) \ \cdots \ w_N(t) \right]^{\top}$. The operator $\C_{\tilde{y}}: L^1_{loc}(\R, [0,1]) \to [0, 1]^{N}$ represents the way the agents are sensing, which ends up being:
\[
    \C_{\tilde{y}(t)}(\rho(t, \cdot)) = \left[ \rho(t, y_1(t)^+) \ \cdots \ \rho(t, y_N(t)^+) \right]^{\top}.
\]


\begin{figure}
	\centering
	\includegraphics[height=4cm]{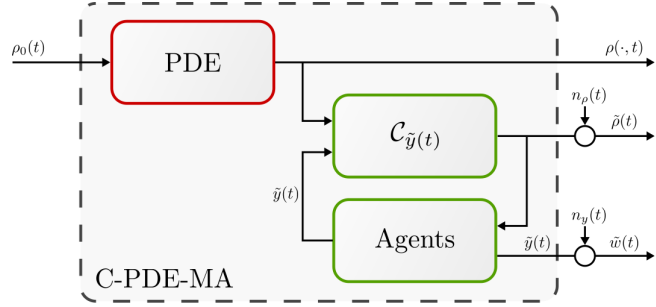}  
	\vspace*{-0.3cm}
	\caption{Block diagram of system~\eqref{eq:C-PDE-MA}.}
	\vspace*{-0.3cm}
	\label{fig:generalProblem2}
\end{figure}




\subsection{Problem formulation}

We are interested in this paper in the reconstruction of the PDE state $\rho$ as expressed by the block diagram in Figure~\ref{fig:generalProblem2}. Before stating the problem, some definitions are given below.


\begin{definition} \label{def:stateReconstruction}
     Let $\Omega \subseteq \R^+$, $(\H, \| \cdot \|_{\H})$ be a semi-normed vector space and $\H_c$ a subset of $H^1(\Omega, \H)$. The \textbf{\mbox{partial-reconstructed} set} $\mathcal{R}$ in $\H_c$ of $\rho \in H^1(\Omega, \H)$ is defined as the set:
     
     \begin{equation} \label{eq:stateReconstruction}
         \mathcal{R} = \Argmin_{\bar{\rho} \ \in \ \H_c} \int_{\Omega} \| \rho(t,\cdot) - \bar{\rho}(t, \cdot) \|_{\H}^2 \ dt.
     \end{equation}
     We say that a signal $\hat{\rho}$ is a \textbf{partial-state reconstruction} (P-SR) if $\hat{\rho} \in \mathcal{R}$.
\end{definition}

\begin{remark}
    Note that this definition encapsulates the following cases:
    \begin{itemize}
        \item If $\Omega = {\infty}$ and $\rho \in \H_c$ then $\hat{\rho}(\infty, \cdot) = \rho(\infty, \cdot)$ almost everywhere and this is the definition of an asymptotic state reconstruction.
        \item If $\H_c$ is a vector space and $\rho \not \in \H_c$, then the P-SR $\hat{\rho}$ is the orthogonal projection of $\rho$ on $\H_c$. 
    \end{itemize}
\end{remark}



The state reconstruction problem implies to choose three entities: the space $\H$ referring to the spatial domain of $\rho(t, \cdot)$; a subset $\H_c$ where we seek $\hat{\rho}$ and the semi-norm $\| \cdot \|_{\H}$ on $\H$. We will construct these three entities in the sequel.

We are working on a time window $\Omega = [0, T]$ (we gather measurements during this time period). For the space interval, we are interested in the reconstruction between the first and the last agents. As noted in \eqref{eq:existenceSpace}, we get $\H = L^1_{loc}(\R, [0,1])$. In the case without noise and if we assume that the model~\eqref{eq:C-PDE-MA} is known, we get that $\rho \in \H_1$ where:
\[
    \H_1 = \!\!\begin{array}[t]{l}
        \left\{ \bar{\rho} \in H^1(\Omega, \H) \ | \ \C_{\tilde{y}(t)} \bar{\rho}(t, \cdot) = \C_{\tilde{y}(t)} \rho(t, \cdot) \text{ for } t \in \Omega \text{ and \eqref{eq:PDE} holds for } (t, x) \in \D_{\Omega} \right\} 
    \end{array} 
\]
with $\D_{\Omega} = \left\{ (t,x) \in \Omega \times \R \ | \ x \in [y_1(t), y_N(t)] \right\}$. Note that $\H_1$ is generally not a vector space. As noted by \cite{klibanov2006estimates,gosse2017filtered}, even if $T \to \infty$, $\H_1$ is usually not a singleton since there is no information on the initial condition. However, there always exists at least one weak solution to \eqref{eq:PDE} so $\H_1$ is not empty.

Adding bias on the measurements implies that $\rho$ does not belong to $\H_1$. Consequently, it is better to use Definition~\ref{def:stateReconstruction} with $\H_2 = \left\{ \bar{\rho} \in H^1(\Omega,\H) \ | \ \text{\eqref{eq:PDE} holds for } (t, x) \in \D_{\Omega} \right\}$ and for $\bar{\rho} \in H^1(\Omega, \H)$:
\begin{equation} \label{eq:continuousSemiNormH}
    \| \bar{\rho}(t, \cdot) \|_{\H}^2 =  \left(\C_{\tilde{y}(t)} \bar{\rho}(t, \cdot)\right)^{\top} \left(\C_{\tilde{y}(t)} \bar{\rho}(t, \cdot)\right).
\end{equation}

It is difficult to solve the optimization problem \eqref{eq:stateReconstruction} since the solution to \eqref{eq:PDE} is usually not smooth \citep{bastin2016stability}. One way to overcome this issue is to consider a slightly different partial differential equation:
\vspace*{-0.5cm}
\begin{equation} \label{eq:viscousPDE}
    \frac{\partial \rho}{\partial t} (t,x) + F(\rho(t,x)) \frac{\partial \rho}{\partial x}(t,x) = \gamma^2 \frac{\partial^2 \rho}{\partial x^2}(t,x),
\end{equation}
where $\gamma \in \R$. This equation has many interesting properties and in particular, the solutions are smoother than in \eqref{eq:PDE}. Using \cite[Theorem~14.6]{amann1993nonhomogeneous} with piece-wise $C^{\infty}$ boundary conditions, then
\begin{equation} \label{eq:hatRho}
    \hat{\rho} \in C^1_p\left((0, T], C^2_p(\R, \R)\right).
\end{equation}
Moreover, the vanishing viscosity method applied to this problem \cite[Lemma~4.2]{coclite2013vanishing} shows that the solution to \eqref{eq:viscousPDE} converges in a weak sense to the entropic solution to \eqref{eq:PDE} as $\gamma$ goes to zero. The set $\H_c(\gamma)$ becomes then:
\begin{equation} \label{eq:Hc}
    \H_c(\gamma) = \left\{ \bar{\rho} \in H^1(\Omega, H^2\left(\R, \R)\right) \ | \ \text{\eqref{eq:viscousPDE} holds for } (t, x) \in \D_{\Omega} \text{ and } \gamma \in \R \right\}.
\end{equation}


Using the previous definitions, we can define the partial-reconstructed set $\mathcal{R}$. 

\vspace{0.1cm}

\textbf{Objective:} Find a P-SR of $\rho$ from \eqref{eq:C-PDE-MA} in $\H_c(\gamma)$ for the semi-norm in \eqref{eq:continuousSemiNormH} with the smallest $\gamma^2$.

\vspace{0.1cm}

Since the set $\H_c(\gamma)$ contains an equality constraint, it might be difficult to solve the optimization problem~\eqref{eq:stateReconstruction} with \eqref{eq:continuousSemiNormH}-\eqref{eq:Hc}. Using the Lagrange multiplier $\lambda_{\F} > 0$, we propose the following relaxation:
\vspace*{-0.5cm}
\begin{multline}\label{eq:relaxOptProblem}
    \mathcal{R}_c(\gamma) = \Argmin_{\bar{\rho} \in H^1\left(\Omega, H^2(\R, \R)\right)} \left\{  \int_{\Omega} \| \rho(t, \cdot) - \bar{\rho}(t, \cdot) \|_{\H}^2 \ dt  + \lambda_{\F} \iint_{\D_{\Omega}} \F_{\gamma}(\bar{\rho}, t, x)^2 \ dx \ dt \right\}
\end{multline}
where $\F_{\gamma}(\bar{\rho}, t,x) = \frac{\partial \bar{\rho}}{\partial t}(t,x) + F(\bar{\rho}(t,x)) \frac{\partial \bar{\rho}}{\partial x}(t,x) - \gamma^2 \frac{\partial^2 \bar{\rho}}{\partial x^2}(t, x)$. It follows from the previous discussion that $\mathcal{R} \subseteq \mathcal{R}_c(\gamma)$ and, in a weak-sense $\mathcal{R} = \mathcal{R}_c(\gamma)$. An element of $\mathcal{R}_c(\gamma)$ is a P-SR of $\rho$ solution to \eqref{eq:C-PDE-MA} in $\H_c(\gamma)$ for the semi-norm~\eqref{eq:continuousSemiNormH}. The objective rewrites then as follows.

\vspace{0.1cm}

\textbf{Rephrased objective:} Find a signal $\hat{\rho} \in \mathcal{R}_c(\gamma)$ with the smallest $\gamma^2$.

\vspace{0.1cm}

To the best of the author knowledge, the P-SR solution to an hyperbolic PDE with noisy measurements coming from a network of agents has never been considered in the literature. We propose to solve this using machine learning as explained in the following section.

\section{Learning-based State Reconstruction}

This section is divided into two subsections. In the first one, we derive a P-SR problem associated to \eqref{eq:C-PDE-MA} without noise. In the second one, the same is discussed with addition of noisy measurements.

\subsection{Neural network solution to the P-SR problem without noise}

A neural network of $N_n$ neurons designed for a regression problem is the mapping $\Theta_i(\vec{X}) = \psi(b_i + \theta_i^{\top} \vec{X})$, where $b_{i}$ and $\theta_{i}$ are the bias and the weight tensors respectively. $\psi$ is the element-wise activation function. 
A densely connected deep-neural network is a neural network with $N_L$ layers, meaning that $\Theta = \Theta_1 \circ \Theta_2 \circ \cdots \circ \Theta_{N_L}$. This neural network is trained on the data-set $(\vec{t}, \vec{x}, \rho(\vec{t}, \vec{x}))$, meaning that the weights and biases are found as the arguments that locally minimizes a loss function $\L$. For regression problems, the loss function is usually chosen to be $\frac{1}{N_{\rm{data}}} \sum_{k=1}^{N_{\rm{data}}} \left( \rho(t_k, x_k) - \Theta(t_k, z_k) \right)^2$ where $N_{\rm{data}} \in \mathbb{N}$ is the number of measurements. This is the Mean-Square Error (MSE) between the observed and estimated outputs. Note that this loss function is a discretization of the $L^2$-norm of the error on the whole space $\D_{\Omega}$. 

We cannot measure points all over the domain but only on the image of the operator $\mathcal{C}_{\tilde{y}}$. For $\{t_k\}_{k \in 1,\dots,N_{\rm{data}}}$ \textit{uniformly distributed} in $[0, T]$, $\{(t_j^{\F}, x_j^{\F})\}_{j \in 1,\dots,N_{\F}}$ \textit{uniformly distributed} in $\D_{\Omega}$ and $\lambda_1, \lambda_2, \lambda_{\gamma} > 0$, an equivalent discretization of the cost function in \eqref{eq:relaxOptProblem} is:
\vspace*{-0.3cm}
\begin{equation} \label{eq:costNN1}
    \L_{\vec{t}}(\tilde{y}, \tilde{\rho}, \Theta) = \frac{\lambda_1}{N_{\rm{data}}} \sum_{k=1}^{N_{\rm{data}}} \left\| \tilde{\rho}(t_k) - \C_{\tilde{y}(t_k)} \Theta(t_k, \cdot) \right\|^2 
    + \frac{\lambda_2}{N_{\F}} \sum_{j=1}^{N_{\F}} \left( \F(\Theta, t_j^{\F}, x_j^{\F}) \right)^2 + \lambda_{\gamma} \gamma^2.
\end{equation}
\vspace*{-0.3cm}
$\L$ is divided into three parts: 
\begin{enumerate}
    \item the first term is the traditional MSE over the measurements.
    \item the second term (often called the physics MSE) does not rely on measurements but only depends on the weights and biases of the neural network, it is derived using automatic differentiation \citep{baydin2017automatic} based on the dynamic of the system. As stressed in \cite{raissi2019physics},it can be seen as a regularization agent, preventing over-fitting by disregarding non physical solutions and reducing the impact of the noise by enforcing the model. The choice of the weight $\lambda_2$ depends then on the confidence we have in the model and the sparsity of the dataset.
    \item the minimization of $\gamma^2$ such that the solutions of \eqref{eq:viscousPDE} are close to the solutions of \eqref{eq:PDE}.
\end{enumerate} 

The second term of the previous objective function was introduced by \cite{lagaris1998artificial,psichogios1992hybrid} where the authors used a neural network to provide a solution to a differential equation. It has been investigated more recently in \cite{raissi2019physics} under the name "Physics-informed deep learning" where it has been extended to robust parameter identification.



Using a backpropagation algorithm, we train the neural network by locally solving the following optimization problem:
\vspace*{-0.1cm}
\begin{equation} \label{eq:Theta}
    \mathcal{R}_{N_n} = \Argmin_{b_{1}, \dots, b_{N_L}, \theta_{1}, \dots, \theta_{N_L}} \L_{\vec{t}}(\tilde{y}, \tilde{\rho}, \Theta).
\end{equation}
Using a neural network to solve this problem is reasonable since, as shown in \cite{sirignano2018dgm}, it has a convergence property recalled below.
\begin{theo} \label{theo:convergence}
    In the case $\lambda_{\gamma} = 0$, the following holds:
    \vspace*{-0.3cm}
    \[
        \lim_{N_n, N_{\rm{data}}, N_{\F} \to \infty} \L_{\vec{t}}(\tilde{y}, \tilde{\rho}, \Theta) = 0 \ \text{ and } \  \lim_{N_n, N_{\rm{data}}, N_{\F} \to \infty} \left( \inf_{\rho_1 \in \mathcal{R}_{N_n}, \rho_2 \in \mathcal{R}_c(\gamma)} \iint_{\D_{\Omega}} (\rho_1 - \rho_2)^2 \right) = 0.
    \]
\end{theo}

The proof is given in Appendix~\ref{app:1}) and is similar to the one of Theorems~7.1 by \cite{sirignano2018dgm}. 
The previous theorem means that $\Theta$ in \eqref{eq:Theta} is an approximation of a P-SR and if the number of neurons tends to infinity, it tends to be a P-SR in $\H_c(\gamma)$. 

\subsection{Neural network solution to the P-SR problem with noise}

First, assume that $n_{\rho}$ is a realization of a Gaussian distribution with mean $\mu$ and standard deviation $\sigma$. The current optimization problem~\eqref{eq:Theta} implies the minimization of a MSE, consequently, the variance of the noise on the agents' trajectories is minimized. The second part of the cost function enforces the signal to be a solution to the noiseless model keeping a low variance as well. Such an implementation performs well and rejects the unbiased noise $n_{\rho}$ \citep{raissi2019physics}. 

In the biased case, if $T$ is large enough, we get $\L_{\vec{t}}(\tilde{y}, \tilde{\rho}, \rho + \varepsilon \mu) = N \left( \sigma^2 + (1-\varepsilon)\mu^2\right) + o_0(\varepsilon)$. Consequently, $\varepsilon$ small and $\mu \neq 0$ yield $\L_{\vec{t}}(\tilde{y}, \tilde{\rho}, \rho + \varepsilon \mu) < \L_{\vec{t}}(\tilde{y}, \tilde{\rho}, \rho)$ and $\rho$ is not a local minimizer. One solution to this issue is to consider instead the loss function $\L_{\vec{t}}(\tilde{y}, \tilde{\rho} - \bar{n}_{\rho}, \Theta)$ where $\bar{n}_{\rho} \in \R^N$.

To deal with the noise $n_y$, the neural network architecture has to be modified. In that case, one need to propose a similar architecture with some physics cost on $\tilde{w}$, leading to the diagram in Figure~\ref{fig:neuralNetwork2}. There are then two neural networks that are trained simultaneously. $\Theta$ is the same as in the previous part and the second one, $\Phi$, estimates $\tilde{y}$. Naturally, the loss function is:
\begin{equation*}
    \tilde{\L}_{\vec{t}}(\tilde{w}, \tilde{\rho}) = \L_{\vec{t}}(\Phi, \tilde{\rho} - \bar{n}_{\rho}, \Theta)
    + \sum_{k=1}^{N_{\rm{data}}} \frac{\lambda_3 \left\| \Phi(t_k) - \tilde{w}(t_k) \right\|^2}{N_{\rm{data}} \times N} 
    + \sum_{l=1}^{N_{\G}} \frac{\lambda_4 \left\| \G(\Phi, t_l^{\G}, \Theta\left( t_l^{\G}, \Phi(t_l^{\G}) \right) \right\|^2}{N_{\G} \times N}
\end{equation*}
where $\L_{\vec{t}}$ is defined in \eqref{eq:costNN1}, $\bar{n}_{\rho} \in \R^N$, $\lambda_3, \lambda_4 > 0$ and $\G(\tilde{y}, t, \tilde{\rho}) = \dot{\tilde{y}}(t) - V(\tilde{\rho}(t))$. 

\begin{remark} The physics cost related to the trajectories of the agents is adding some artificial boundary data. Consequently, the measurements can be sparser for a similar trajectory estimation. 
\end{remark}

\vspace*{-0.5cm}

\begin{figure}
	\centering
	\includegraphics[width=0.9\textwidth]{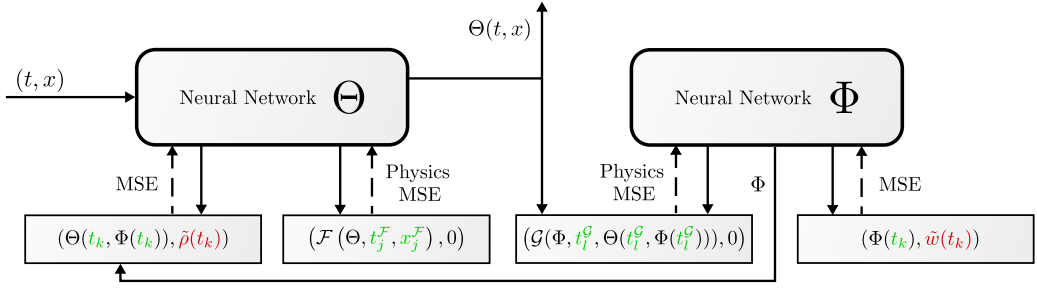}  
	\caption{Block diagram of the neural network for P-SR with noise. The dashed arrows refer to values used for training the neural network. Green variables are a priori choices while the red ones are the measurements.}
	\vspace*{-0.6cm}
	\label{fig:neuralNetwork2}
\end{figure}
%

\subsection{Numerical implementation}

\subsubsection{Network architecture}

It is well-known that some solutions of \eqref{eq:PDE} are discontinuous while solutions of \eqref{eq:viscousPDE} are usually not. Nevertheless, they may have fast variations. The neural network $\Theta$ must handle both shock and rarefaction waves while $\Phi$ represents an absolutely continuous function.

$\Phi$ takes one input (the time $t$) and has $N$ outputs (the positions of the agents). Since the solution is absolutely continuous (it is piece-wise $C^1$), the best is to combine some $\ReLu$ and $\tanh$ activation functions. To keep a simple structure, one can then consider two dense neural networks with different activation functions such as $\ReLu$ and $\tanh$. They are summed and weighed at the end to give the final output. There are $3$ hidden layers, each with $2N$ neurons for each sub-network. 

For $\Theta$, the architecture is more complex. Since we are considering a viscous PDE, the solutions are continuous and using the $\tanh$ as an activation function seems a good choice for several reasons:
\begin{enumerate}
    \item A dense neural network with $\tanh$ activation functions can deal with almost-shock waves since $\Theta(t,x) = \tanh \left( k T_s(t,x) \right)$ where $T_s$ is the trajectory of a shock wave. If $k$ is chosen very large, then there is a very steep slope along the curve $T_s(t,x) = 0$.
    \item It can deal with rarefaction waves as well since in this case the solution is $F^{-1}(x / t)$ \cite[Section~11.2]{evans1998partial} which is continuous since $F$ is $C^1$ and $F' < 0$. 
\end{enumerate}
Consequently, the number of nodes can be related to the number of waves in the solution. On the contrary, the number of layers deals with the inference between these waves. In conclusion, the number of nodes is related to the length-space while the number of layers reflects the time-window. \cite{choromanska2015loss} showed that having a neural network with many variables usually leads to better convergence. As said previously, the neural network should not overfit the data because of the regularization agent so we can consider a large network with $8 T /100$ layers made up of $20 L / 7000$ neurons each. 

\subsubsection{Training procedure}

First of all, since the space and time intervals are not of the same order of magnitude, it is better to standardize the dataset before training. Indeed, that leads to a faster convergence of the optimization procedure to a more optimal solution \citep{Goodfellow-et-al-2016}.

Similarly to what was done in \cite{raissi2019physics,lagaris1998artificial,sirignano2018dgm}, the best solver for this problem seems to be the BFGS algorithm \cite[Chap 6]{nocedal2006numerical}. Nevertheless, it is quite slow so one can consider speeding it up with a pre-training using Adam \citep{kingma2015adam}. Another issue comes with the choice of the different $\lambda$. It appears that the effect of this choice decreases when we split the training into several steps, leading to the following procedure:
\begin{enumerate}
    \item Let $\lambda_1 = \lambda_2 = \lambda_{\gamma} = 0$ and $\lambda_3 = 1$, $\lambda_4 = 0.5$. We then estimate precisely the trajectories of the agents and we modify the value of the density along these curves. We train using the BFGS algorithm with low requirements such that the convergence is rapidly obtained.
    \item Then, we train the value of the density without modifying the trajectories. For that, we fix the $\Phi$ neural network and we use first Adam and then BFGS with $\lambda_1 = 1$, $\lambda_2 = 0.1$, $\lambda_{\gamma} = \lambda_3 = 0$ and $\lambda_4 = 0.25$. This step takes more time and fixes the values of the density at its boundaries while starting to enforce the physics of the PDE inside the domain.
    \item Finally, we optimize using BFGS only and $\lambda_1 = 1$, $\lambda_2 = 1$, $\lambda_{\gamma} = 0.1$, $\lambda_3 = 1$ and $\lambda_4 = 0.5$. We focus on this step on the noise reduction and we minimize the coupled system proposed in Figure~\ref{fig:neuralNetwork2}. Since the trajectories are close to their optimal values, this step changes mostly the values of the density inside the domain, trying to keep a low value of $\gamma$. 
\end{enumerate}
This procedure seems, in general, to be a good starting point for such a problem since it follows the iterative procedure described in \cite{bertsekas2014constrained}. The weights values have to be designed in each case and the one used here leads to the lowest computed generalization error in our case.

\section{Traffic flow theory and simulation results}
\label{sec:trafficFlow}

\subsection{Traffic flow theory}

The simplest model for traffic flow on a highway is given by the continuity equation, which is nothing more than the conservation of mass \cite[Chapter~11]{evans1998partial}. This model has been derived by  \cite{lighthill1955kinematic,richards1956shock} and it results in the so-called LWR model:
\begin{equation} \label{eq:continuity}
        \frac{\partial \rho}{\partial t}(t,x) + \frac{\partial \rho}{\partial x} (\rho V)(t, x) = 0, \quad \quad (t,x) \in \R^+ \times \R
\end{equation}
where $\rho: \R^+ \times \R \to [0, 1]$ is the normalized density of cars. It can be represented as the probability distribution of a car at a given position $x$ at time $t$. $V: [0,1] \to \R^+$ stands for the equilibrium velocity of a car depending on the density around. The LWR model assumes, similarly to \cite{greenshields1935study}, that $V = V_f (1-\rho)$, where $V_f > 0$ is the free flow speed. 
%
%
With this choice, \eqref{eq:continuity} transforms into the following scalar hyperbolic PDE for $(t,x) \in \R^+ \times \R$:
\begin{equation} \label{eq:LWR}
        \frac{\partial \rho}{\partial t} (t,x) + V_f \left( 1 - 2\rho(t, x) \right) \frac{\partial \rho}{\partial x}(t, x) = 0.
\end{equation}

In the context of traffic flow, we are considering that there are some vehicles -- called Probe Vehicles (PV) -- capable of sensing the density around them using the distances to their neighbors, cameras \citep{chellappa2004vehicle,9081940} or radar sensors \citep{1608034}. The state reconstruction problem without noise is solved in finite-time by \cite{delle2019traffic,barreau2020dynamic} for instance. The results in our case are shown in the following subsection. A full study on this example is conducted in \cite{liu2020learning}.

\subsection{Simulation results}
\label{sec:results}

We are studying here the P-SR problem of \eqref{eq:LWR} using probe vehicles. The real density is obtained from a Godunov scheme \citep{leveque1992numerical} with random initial and boundary conditions. One can see in Figure~\ref{fig:godunov} that shocks and rarefaction waves are making the reconstruction difficult. We did the reconstruction in two cases with a neural network coded using Tensorflow~1.15 \citep{tensorflow2015-whitepaper,github}. The first scenario considers the system without noise; the result is displayed in Figure~\ref{fig:error}. The second case contemplates $n_{\rho}$ as an unbiased Gaussian noise with a standard deviation of $\sigma_{\rho} = 0.2$ and $n_y$ as a Brownian motion generated from a normal distribution with zero mean and a variance $\sigma_{y} = 2$ (see Figure~\ref{fig:errorNoise}). $20$ P-SR are done using a naive BFGS \cite{raissi2019physics,sirignano2018dgm} and $20$ others using the training procedure described earlier. The computations are done on a Intel Core i5-8365 CPU @ 1.60GHz and 16 GB of RAM computer. The computation time and the normalized error are displayed in Figure~\ref{fig:stats}.


\begin{figure}
    \centering
    \includegraphics[width=0.5\linewidth]{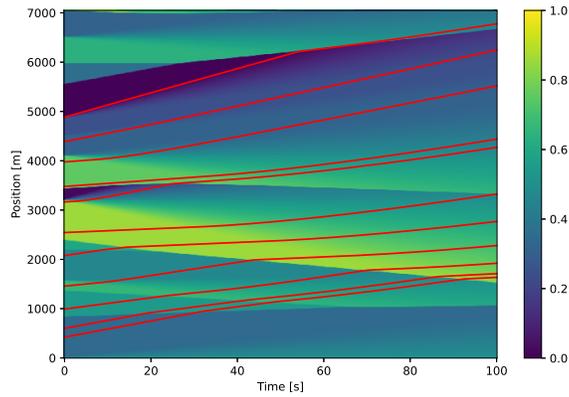}
    \vspace*{-0.5cm}
    \caption{Real density obtained using the Godunov scheme with random initial and boundary conditions. The red lines are the trajectories of the PV.}
    \vspace*{-0.6cm}
    \label{fig:godunov}
\end{figure}

\begin{figure*}
\floatconts
  {fig:rho_example}
  {\vspace*{-0.7cm}\caption{Absolute value of the error between the P-RS and the real one. The red lines refer to the real trajectories of the PV and the orange ones are the reconstructed ones.}}
  {%
    \subfigure[Scenario without noise]{\label{fig:error}%
      \includegraphics[width=0.4\linewidth]{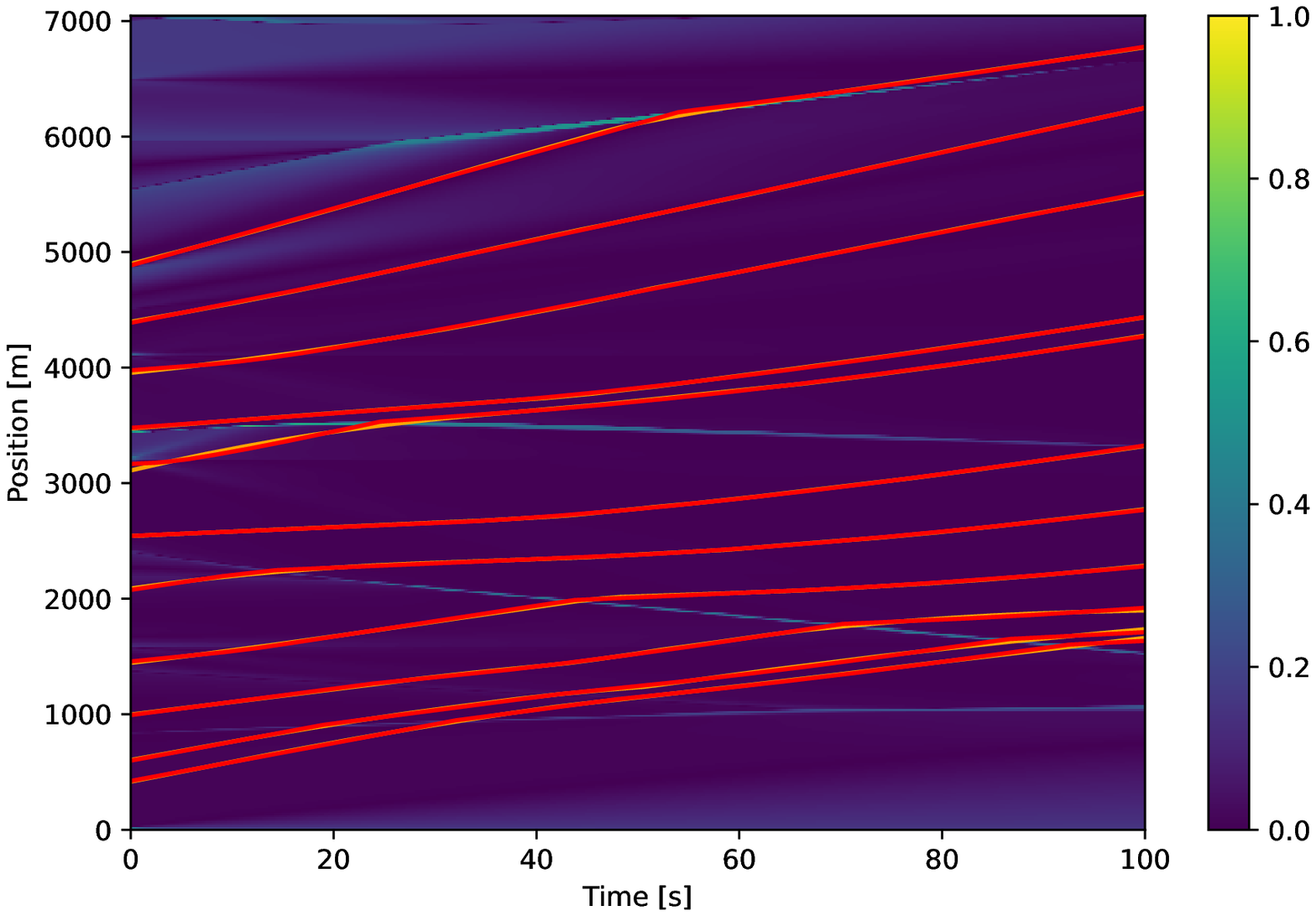}}%
    \quad
    \subfigure[Scenario with noise]{\label{fig:errorNoise}%
      \includegraphics[width=0.4\linewidth]{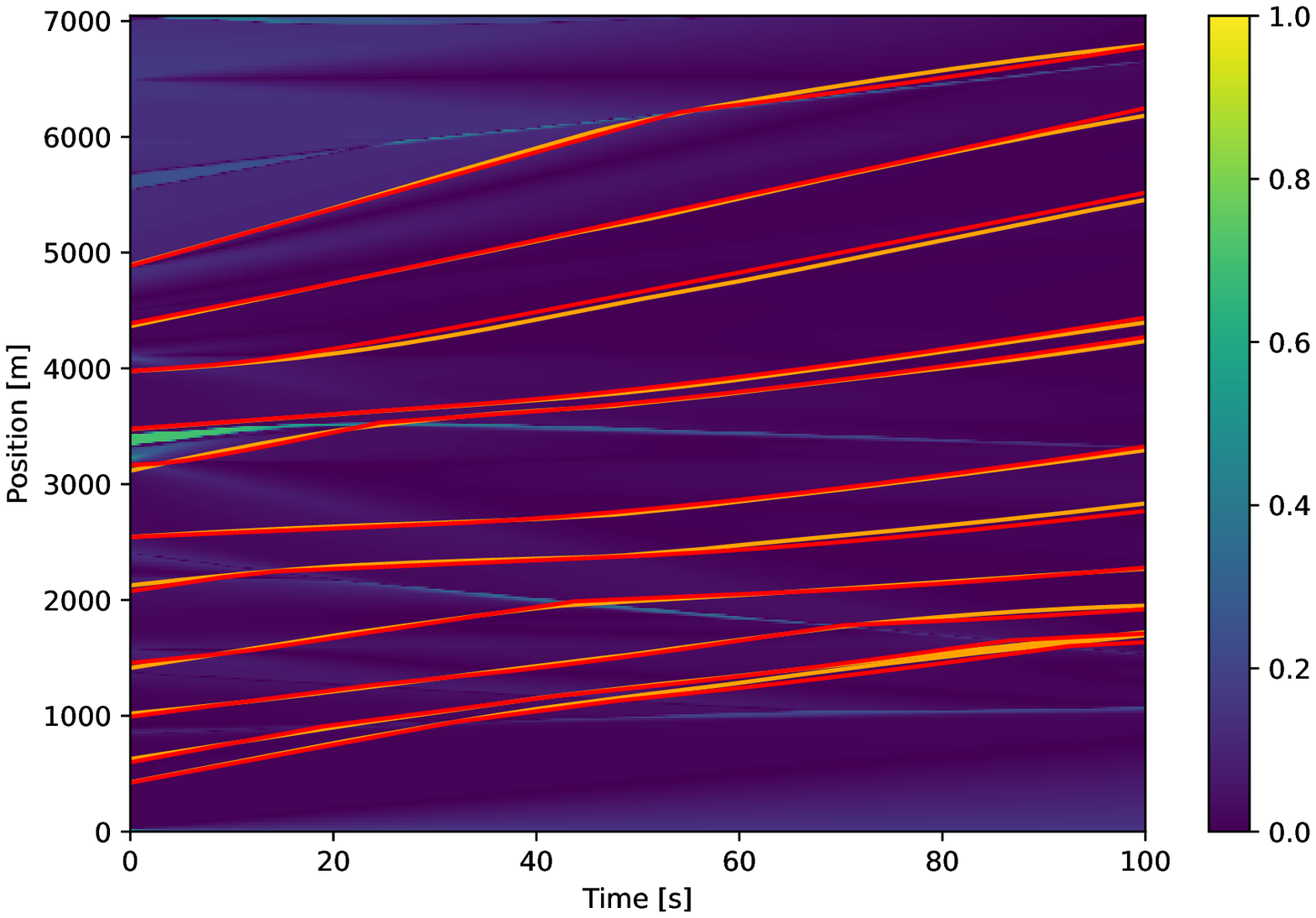}}%
  }
  \vspace*{-0.6cm}
\end{figure*}

\begin{figure*}
\hspace*{-1cm}
\floatconts
  {fig:stats}
  {\vspace*{-0.7cm}\caption{Computation time and error depending on the training algorithm.}}
  {%
    \subfigure[Computation time in seconds]{\includegraphics[width=0.4\linewidth]{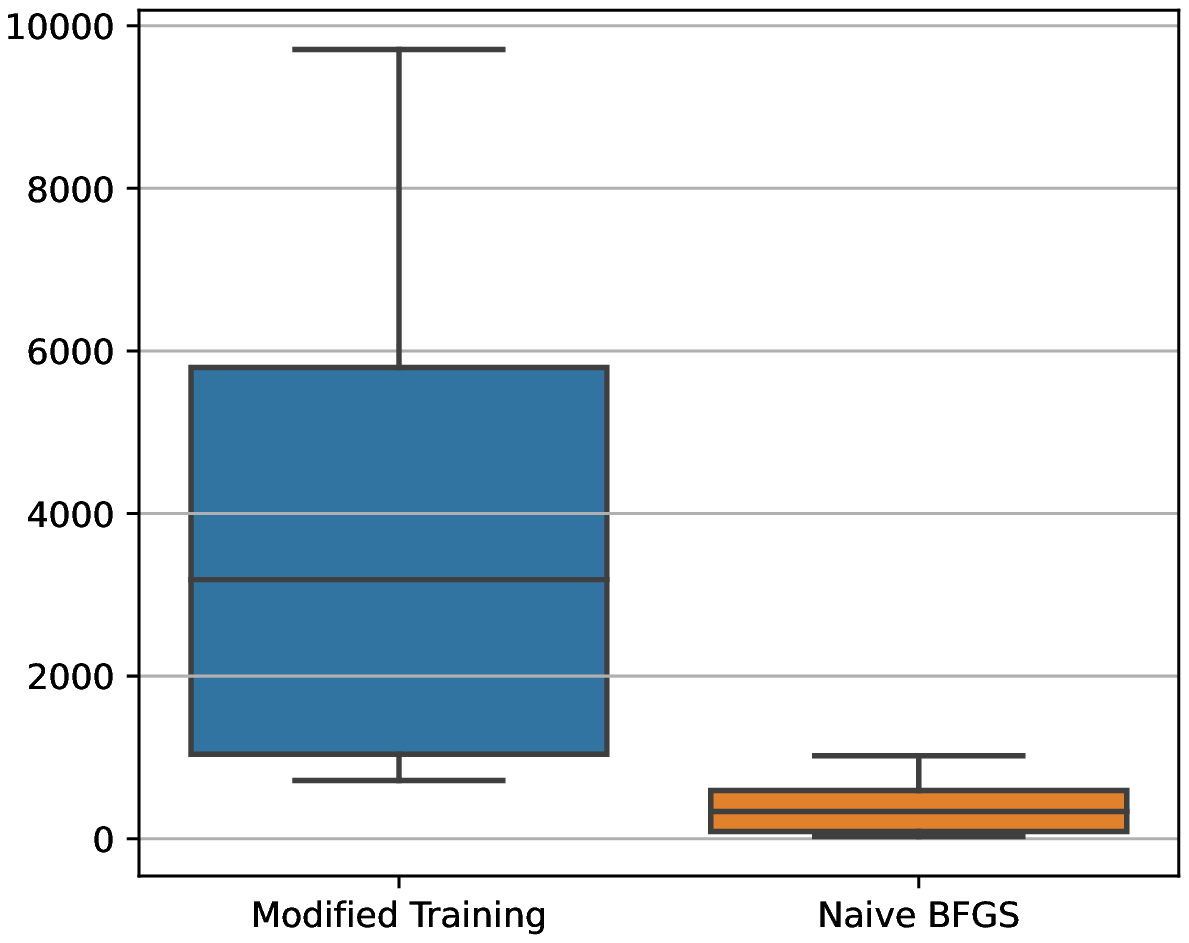}}%
    \quad \quad
    \subfigure[Normalized generalization error]{\includegraphics[width=0.4\linewidth]{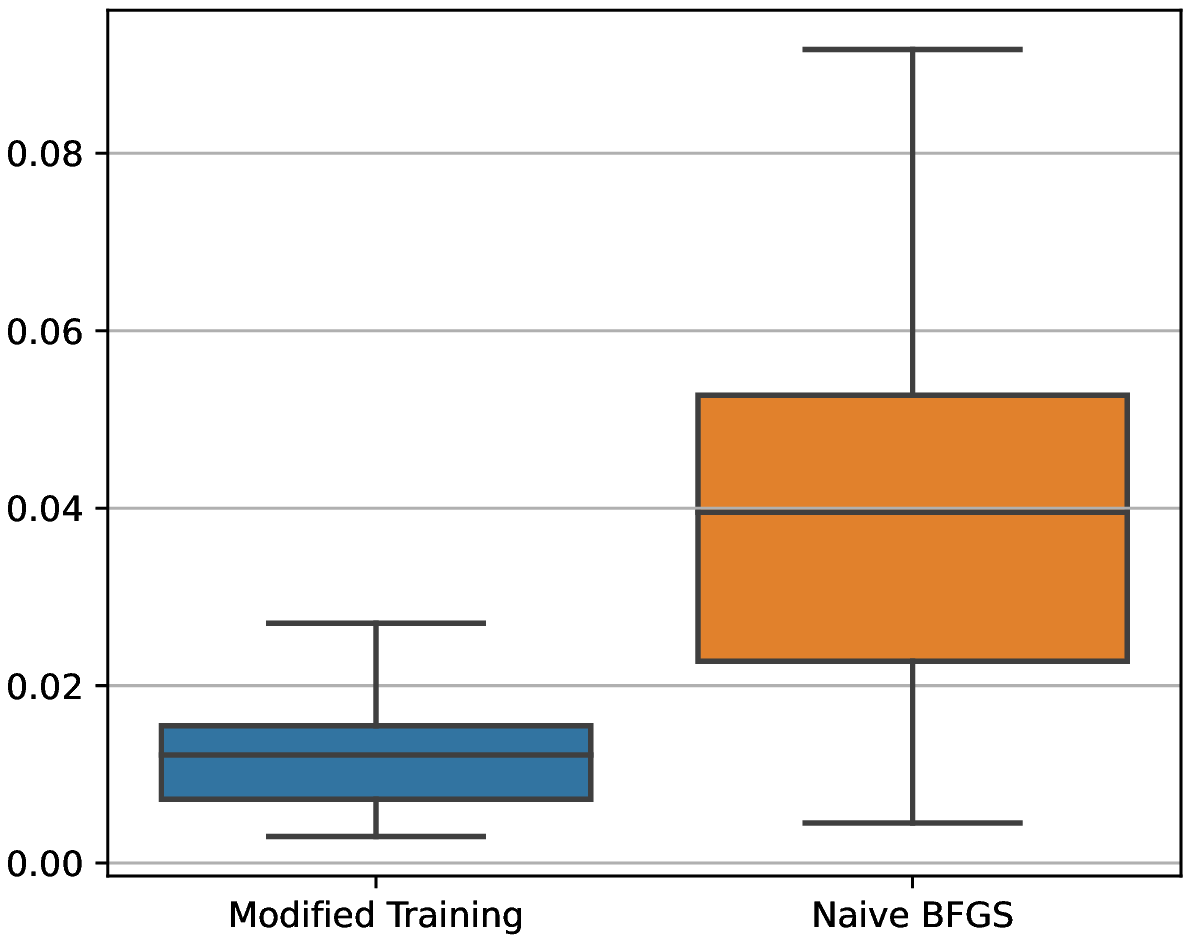}}%
  }
\end{figure*}

One can see in Figure~\ref{fig:error} that the reconstruction in the first case is almost perfect. There are some discrepancies for small-time $t$ because of the effect of the unknown initial condition. But, after a while, the reconstruction error is close to zero between the PV except on the shock trajectories where the neural network is unable of reproducing the steep variation. In the second case (see Figure~\ref{fig:errorNoise}), the error is larger. The PV trajectories are not as sharp as in the first case because of the noise correction. Nevertheless, the P-SR is very close to the one obtained previously, showing that the proposed solution is indeed capable of reducing the noise in the measurements. The modified training has a larger computation time ($10$ times larger in average) but a smaller and more consistent generalization error ($4$ times smaller in average with a standard deviation is $3$ times smaller). That means the modified training procedure is indeed more accurate and robust to the neural network initialization at the price of a much larger computation time.


\section{Conclusion \& Perspectives}

In this paper, we proposed the partial-state reconstruction problem of a hyperbolic equation using only Lagrangian measurements and we applied it to the special example of traffic flow estimation. To that extend, we used physics-informed deep learning to counterbalance the small number of measurements and increase the robustness. The simulations showed that noise was indeed rejected.

This is preliminary work and we propose many perspectives. First of all, the neural network can be improved by considering a more specific architecture \citep{hochreiter1997long}. One can also look into the direction of a decentralized reconstruction algorithm, that would effectively reduce the computational burden while increasing the sensitivity to noise. The roles of the $\lambda$ have not been explored in the paper, however, it appears that the speed of convergence highly depends on these values. A theoretical study on the role of these multipliers should be done to get a better understanding of the optimization process. Finally, one can also consider second-order models which are more adapted to traffic flow \citep{aw1995,zhang2002}. 

\acks{The research leading to these results is partially funded by the KAUST Office of Sponsored Research under Award No. OSR-2019-CRG8-4033, the Swedish Foundation for Strategic Research, the Swedish Research Council and Knut and Alice Wallenberg Foundation.}

\appendix

\section{Proof of Theorem~\ref{theo:convergence}}
\label{app:1}

First of all, note that for a number of physical and measurement points sufficiently large, $\L_{\vec{t}}(\tilde{y}, \tilde{\rho}, \Theta)$ with $\lambda_{\gamma} = 0$ is arbitrary close to $J(\rho, \Theta)$ defined as:
\[
    J(\rho, \bar{\rho}) = \int_{\Omega} \| \rho(t, \cdot) - \bar{\rho}(t, \cdot) \|_{\H}^2 \ dt  + \lambda_{\F} \iint_{\D_{\Omega}} \F_{\gamma}(\bar{\rho}, t, x)^2 \ dx \ dt.
\]
We will then consider the minimization of $J$ instead of the minimization of $\L_{\vec{t}}$ in the sequel.

We follow a proof similar to the one of Theorem~7.1 by \cite{sirignano2018dgm}. Let us first denote by $\hat{\rho}_i$ for $i \in \{1, \dots, N-1\}$ the solutions of the following systems for $t \in [0, T]$:
\begin{equation} \label{eq:heatEq}
    \left\{
        \begin{array}{ll}
            \displaystyle \frac{\partial \hat{\rho}_i}{\partial t} (t,x) + F(\rho(t,x)) \frac{\partial \hat{\rho}_i}{\partial x}(t,x) = \gamma^2 \frac{\partial^2 \hat{\rho}_i}{\partial x^2}(t,x), & x \in [y_i(t), y_{i+1}(t)], \\
            \displaystyle \hat{\rho}_i(t, y_i(t)) = \rho(t, y_i(t)), \ \hat{\rho}_i(t, y_{i+1}(t)) = \rho(t, y_{i+1}(t)), \\
            \hat{\rho}^0 \in C^{\infty}([y_i(0), y_{i+1}(0), \R]).
        \end{array}
    \right.
\end{equation}
We want to show that the minimum value of $J(\rho, \Theta)$ when $N_n \to \infty$ is zero. Since $\F_{\gamma}(\hat{\rho}, t, x) = 0$, we then get
\[
    \begin{array}{rl}
        \displaystyle J(\rho, \Theta) \!\!\!&\displaystyle= \int_{\Omega} \| \rho(t, \cdot) - \Theta(t, \cdot) \|_{\H}^2 \ dt  + \lambda_{\F} \iint_{\D_{\Omega}} \F_{\gamma}(\Theta, t, x)^2 \ dx \ dt \\
        &\displaystyle= \int_{\Omega} \| \hat{\rho}(t, \cdot) - \Theta(t, \cdot) \|_{\H}^2 \ dt  + \lambda_{\F} \iint_{\D_{\Omega}} \left( \F_{\gamma}(\Theta, t, x) - \F_{\gamma}(\hat{\rho}, t, x) \right)^2 \ dx \ dt \\
        &\displaystyle= \int_{\Omega} \| \hat{\rho}(t, \cdot) - \Theta(t, \cdot) \|_{\H}^2 \ dt  + \lambda_{\F} \iint_{\D_{\Omega}} \left( \F_{\gamma}(\Theta, t, x) - \F_{\gamma}(\hat{\rho}, t, x) \right)^2 \ dx \ dt \\
        &\displaystyle\leq \int_{\Omega} \| \hat{\rho}(t, \cdot) - \Theta(t, \cdot) \|_{\H}^2 \ dt  + 2 \lambda_{\F} \iint_{\D_{\Omega}} \left( \Theta_t(t, x) - \hat{\rho}_t(t, x) \right)^2 \ dx \ dt \\
        &\displaystyle \hfill + 2 \lambda_{\F} \iint_{\D_{\Omega}} \left( F(\Theta(t, x)) \Theta_x(t, x) - F(\hat{\rho}(t, x)) \hat{\rho}_x(t, x) \right)^2 \ dx \ dt \\
        &\displaystyle \hfill + 2 \gamma^4 \lambda_{\F} \iint_{\D_{\Omega}} \left( \Theta_{xx}(t, x) - \hat{\rho}_{xx}(t, x) \right)^2 \ dx \ dt \\
        &\displaystyle\leq \int_{\Omega} \| \hat{\rho}(t, \cdot) - \Theta(t, \cdot) \|_{\H}^2 \ dt  + 2 \lambda_{\F} \iint_{\D_{\Omega}} \left( \Theta_t(t, x) - \hat{\rho}_t(t, x) \right)^2 \ dx \ dt \\
        &\displaystyle \hfill + 4 \lambda_{\F} \iint_{\D_{\Omega}} \left( F(\Theta(t, x)) - F(\hat{\rho}(t, x)) \right)^2 \hat{\rho}_x(t, x)^2 \ dx \ dt \\
        &\displaystyle \hfill + 4 \lambda_{\F} \iint_{\D_{\Omega}} F(\Theta(t, x))^2 \left( \Theta_x(t, x) - \hat{\rho}_x(t, x) \right)^2 \ dx \ dt \\
        &\displaystyle \hfill + 2 \gamma^4 \lambda_{\F} \iint_{\D_{\Omega}} \left( \Theta_{xx}(t, x) - \hat{\rho}_{xx}(t, x) \right)^2 \ dx \ dt \\
    \end{array}
\]
The last inequalities are obtained using Young inequality. Since $\hat{\rho}$ lies in a compact set and $F \in C^1(\R, \R)$, there exists $M > 0$ such that $| F(\Theta) - F(\hat{\rho}) |^2 \leq M | \Theta - \hat{\rho} |^2$. That yields:
\[
    \begin{array}{rl}
        \displaystyle J(\rho, \Theta) \!\!\!&\displaystyle\leq \int_{\Omega} \| \hat{\rho}(t, \cdot) - \Theta(t, \cdot) \|_{\H}^2 \ dt  + 2 \lambda_{\F} \iint_{\D_{\Omega}} \left( \Theta_t(t, x) - \hat{\rho}_t(t, x) \right)^2 \ dx \ dt \\
        &\displaystyle \hfill + 4 \lambda_{\F} M^2 \left( \max_{(t, x) \in \D_{\Omega}} |\hat{\rho}_x(t, x)|^2 \right) \iint_{\D_{\Omega}} \left( \Theta(t, x) - \hat{\rho}(t, x) \right)^2 \ dx \ dt \\
        &\displaystyle \hfill + 4 \lambda_{\F} \left( \max_{(t, x) \in \D_{\Omega}} F(\Theta(t, x))^2 \right) \iint_{\D_{\Omega}} \left( \Theta_x(t, x) - \hat{\rho}_x(t, x) \right)^2 \ dx \ dt \\
        &\displaystyle \hfill + 2 \gamma^4 \lambda_{\F} \iint_{\D_{\Omega}} \left( \Theta_{xx}(t, x) - \hat{\rho}_{xx}(t, x) \right)^2 \ dx \ dt \\
    \end{array}
\]
The regularity of $\hat{\rho}$ expressed by \eqref{eq:hatRho} implies that Theorem~3 by \cite{hornik1991approximation} holds. That yields for any $\varepsilon > 0$, there exists $N_n$ sufficiently large such that
\begin{multline*}
    \hspace*{-0.4cm}\max_{(t, x) \in \D_{\Omega}} \left| \frac{\partial \hat{\rho}}{\partial t}(t,x) - \frac{\partial \Theta}{\partial t}(t,x) \right| + \max_{(t, x) \in \D_{\Omega}} \left| \hat{\rho}(t,x) - \Theta(t,x) \right| \\
    \quad \quad \quad + \max_{(t, x) \in \D_{\Omega}} \left| \frac{\partial \hat{\rho}}{\partial x}(t,x) - \frac{\partial \Theta}{\partial x}(t,x) \right| + \max_{(t, x) \in \D_{\Omega}} \left| \frac{\partial^2 \hat{\rho}}{\partial x^2}(t,x) - \frac{\partial^2 \Theta}{\partial x^2}(t,x) \right| \leq \varepsilon.
\end{multline*}
That leads to
\[
    \displaystyle J(\hat{\rho}, \Theta) \leq |\D_{\Omega}| \lambda_{\F} \left\{ \frac{1}{|\lambda_{\F} \D_{\Omega}|} + 2  + 4 M^2 \left( \max_{(t, x) \in \D_{\Omega}} |\hat{\rho}_x(t, x)|^2 \right) + 4 \left( \max_{(t, x) \in \D_{\Omega}} F(\Theta(t, x))^2 \right) + 2 \gamma^4 \right\} \varepsilon^2.
\]
Because of the regularity properties of $\hat{\rho}$ (see \eqref{eq:hatRho}) and $\Theta$, we indeed get that $\lim_{N_n \to \infty} J(\rho, \Theta) = 0$. That means in particular that $\F_{\gamma}(\Theta, t, x) = 0$ almost everywhere and $\Theta \in \H_c(\gamma)$.

\begin{remark}
    Considering Theorem~3 of \cite{sirignano2018dgm}, one can even show that $\lim_{N_n \to \infty} \Theta$ is a weak-solution to \eqref{eq:heatEq}.
\end{remark}

\bibliography{biblio}

\end{document}